\documentclass[12pt]{article}
\usepackage{amsmath,amsthm,amssymb,mathtools}
\usepackage{graphicx}

\def\DP{{}^{\texttiny{C}}_0\! \mathcal{D}}
\def\Arg{\text{Arg}}
\def\Rset{{\mathbb{R}}}
\def\Cset{{\mathbb{C}}}

\def\Nset{{\mathbb{N}}}
\def\eu{{\ensuremath{\mathrm{e}}}}
\def\iu{{\ensuremath{\mathrm{i}}}}
\def\du{{\ensuremath{\mathrm{d}}}}
\def\texttiny#1{{\text{\tiny{#1}}}}

\newtheorem{theorem}{Theorem}
\newtheorem{lemma}{Lemma}
\newtheorem{proposition}{Proposition}

\newtheorem{remark}{Remark}

\begin{document}
\thispagestyle{empty}
\begin{center}

{\bf\Large Stability of fractional-order systems with Prabhakar derivatives\footnote{This research was funded by the COST Action CA 15225 - ``Fractional-order systems- analysis, synthesis and their importance for future design". The work of R. Garrappa was also partially supported by a GNCS-INdAM 2020 Project.}}

\vspace{8mm}

{\large Roberto Garrappa$^1$,\  Eva Kaslik$^{2}$}

\vspace{3mm}

{\em $^1$Department of Mathematics, University of Bari\\ Via E. Orabona 4, 70126 Bari, Italy\\ Member of the INdAM Research Group GNCS, Italy\\         
E-mail: roberto.garrappa@uniba.it\\

$^2$Department of Mathematics and Computer Science\\ 
West University of Timi\c{s}oara\\ 
Bd. V. P\^{a}rvan 4, 300223 Timi\c{s}oara, Romania\\
E-mail: eva.kaslik@e-uvt.ro}

\vspace{5mm}
\end{center}
\vspace{8mm}

\noindent\textbf{Abstract.}
Fractional derivatives of Prabhakar type are capturing an increasing interest since their ability to describe anomalous relaxation phenomena (in dielectrics and other fields) showing a simultaneous nonlocal and nonlinear behaviour. In this paper we study the asymptotic stability of systems of differential equations with the Prabhakar derivative, providing an exact characterization of the corresponding stability region. Asymptotic expansions (for small and large arguments) of the solution of linear differential equations of Prabhakar type and a numerical method for nonlinear systems are derived. Numerical experiments are hence presented to validate theoretical findings.\\

\noindent\textbf{Keywords:} Fractional calculus; fractional Prabhakar derivative; asymptotic stability; stability region.

\section{Introduction}

The Prabhakar function is named after the Indian mathematician Tilak Raj Prabhakar who introduced  in 1971 a generalization to three parameters of the Mittag-Leffler function \cite{GorenfloKilbasMainardiRogosin2014} and studied a convolution integral operator with this function as kernel \cite{Prabhakar1971}. 

After their introduction, Prabhakar's function and integral have been overlooked for a long time until, in the first years of the twenty-first century, the connections with the Havriliak-Negami (HN) dielectric model \cite{HavriliakNegami1966} have been put in light. The HN model was introduced to incorporate the asymmetry and broadness observed in the dielectric dispersion of some polymers and today it is recognized as manifestation of the simultaneous nonlocality and nonlinearity \cite{Miskinis2009,StephanovichGlinchukHilczerKirichenko2002} in the response of complex and heterogeneous systems. For these reasons operators based on the Prabhakar function are employed to describe in the time-domain sophisticated  relaxation models in several areas (e.g., see \cite{BiaCaratelliMesciaCicchettiMaionePrudenzano2014,CausleyPetropoulos2013,colombaro2018storage,GarrappaMainardiMaione2016,GarrappaMaione2017,giusti2018comment,Giusti2020,GiustiColombaro2018,Gorska2019,KhamzinNigmatullinPopov2014,Machado2015,Sandev2017,ZhaoSun2019}).

In 2002 the Prabhakar integral was studied in the context of weakly-singular Volterra integral equations and an interpretation in the framework of fractional calculus was provided \cite{KilbasSaigoSaxena2002}, thus leading two years later to the proposition of a left-inverse operator of the Prabhakar fractional integral \cite{KilbasSaigoSaxena2004}. A regularization of this inverse, known as the fractional Prabhakar derivative, was introduced in \cite{DovidioPolito2018} and one year later all these preliminary ideas were incorporated in a more general theory \cite{GarraGorenfloPolitoTomovski2014}, successively deepened in \cite{GarraGarrappa2018_CNSNS} and \cite{Garrappa2016_CNSNS}. We refer to the recent survey paper \cite{GiustiColombaroGarraGarrappaPolitoPopolizioMainardi2020} for a comprehensive history and  collection of background material and applications of the \emph{Prabhakar fractional calculus}.


Theoretical aspects of the Prabhakar  derivative have been studied in a fair number of works. However, there still persist some not completely clear aspects which must be deepened in order to profitably employ the fractional Prabhakar derivative in the analysis and simulation of linear and nonlinear systems. 

This paper focuses on the asymptotic stability of fractional-order systems with Prabhakar derivatives. Due to the nonlinear dependence of this derivative on a certain number of parameters, this is a difficult and highly complex task, and there are only a couple of previously published papers which have tackled this issue \cite{alidousti2020stability,derakhshan2016asymptotic}, obtaining some sufficient conditions for the asymptotic stability of linear systems with Prabhakar derivative. Consequently, our aim is to clarify several aspects presented in \cite{alidousti2020stability,derakhshan2016asymptotic} and to give a rigorous and complete characterisation of the stability region of fractional-order systems with Prabhakar derivatives, essentially obtaining a generalization of the well known Matignon theorem \cite{Matignon} for standard fractional calculus. 

Our main result is formulated as a necessary and sufficient condition for the asymptotic stability of a linear autonomous systems with Prabhakar derivatives and an application to the study of nonlinear systems is also provided.


This paper is organized in the following way. Section \ref{S:Background} is devoted to present a short review of the basic material on the Prabhakar function and on the fractional Prabhakar calculus. Section \ref{S:AsymptoticStability} describes the the main results concerning stability properties of systems of differential equations with the fractional Prabhakar derivative. A characterisation of the corresponding stability region by means of the root locus method is presented in Section \ref{S:StabilityRegion}. In Section \ref{S:Solution} we derive the asymptotic expansion of the solution of linear differential equation with the Prabhakar derivative together with a numerical method for solving nonlinear problems. Finally, some numerical experiments are presented in Section \ref{S:NumericalExperiments} with the aim of validating the theoretical results.

\section{Preliminaries on Prabhakar  function and Prabhakar calculus}\label{S:Background}

Given three real parameters $\alpha$, $\beta$ and $\gamma$, the Prabhakar function is defined by the series representation
\[
	E_{\alpha,\beta}^{\gamma} (z) = \frac{1}{\Gamma(\gamma)} \sum_{k=0}^{\infty} \frac{\Gamma(\gamma+k) z^{k}}{k! \Gamma(\alpha k + \beta)} 
	, \quad \alpha > 0 , \quad z \in \Cset ,
\]
where, as usual, $\Gamma(x) = \int_0^\infty t^{x-1} \eu^{-t} \du t$ is the Euler-Gamma function. This is an entire function of order $\rho=1/\alpha$ and type $\sigma=1$.

More generally, the Prabhakar function is defined for complex parameters, provided that $\Re(\alpha)>0$; in this paper we prefer however to focus just on real parameters in view of their wider range of applications.

It is immediate to see that when $\gamma=1$ the function $E_{\alpha,\beta}^{\gamma} (z)$ reduces to the standard two parameter ML function $E_{\alpha,\beta}(z)$, when $\beta=\gamma=1$ to the one-parameter ML function $E_{\alpha}(z)$ and when $\alpha=\beta=\gamma=1$ the correspondence with the exponential function $\eu^{z}$ is obtained. Whenever $\gamma=-j$, with $j \in \Nset$, it is easy to verify that the Prabhakar function is the $j$-th degree polynomial 
\[
	E_{\alpha,\beta}^{-j}(z) = \sum_{k=0}^{j} (-1)^k\binom{j}{k} \frac{z^{k}}{\Gamma(\alpha k + \beta)} .
\]

Although an analytical representation of the Laplace transform (LT) of $E_{\alpha,\beta}^{\gamma}(z)$ is not known, it is possible to evaluate the LT of the generalization 
\[
	e_{\alpha,\beta}^{\gamma}(t; \omega) = t^{\beta-1} E_{\alpha,\beta}^{\gamma}(t^{\alpha} \omega)
	, \quad t > 0, \quad \omega \in \Cset,
\]
which, for $\Re(s)> 0$ and $|s| > |\omega|^{\frac{1}{\alpha}}$, is
\[
	{\mathcal E}_{\alpha,\beta}^{\gamma}(s; \omega) \coloneqq {\mathcal L} \Bigl( e_{\alpha,\beta}^{\gamma}(t; \omega) \, ; \, s \Bigr) = \frac{s^{\alpha\gamma - \beta}}{(s^{\alpha} - \omega)^{\gamma}}  .
\]

Since the Prabhakar function, and in particular its generalization $e_{\alpha,\beta}^{\gamma}(t; \omega)$, is employed for the description of  relaxation phenomena, it is of importance to identify the range of parameters for which it turns out to be completely monotonic (CM). We recall that a function $f:(0,+\infty) \to \Rset$ is CM if it has derivatives of any order $k \in \Nset$ and $(-1)^k f^{(k)}(t) \ge 0$ on $(0,+\infty)$. The CM properties of the Prabhakar function have been studied in \cite{MainardiGarrappa2015_JCP,TomovskiPoganySrivastava2014} and it is possible to prove that $e_{\alpha,\beta}^{\gamma}(t; \omega)$ is CM if
\begin{equation}
\label{cond.complete.monot}
   \omega<0,\quad 0 < \alpha \le 1, \quad 0 < \alpha \gamma \le \beta \le 1.
\end{equation}

The asymptotic behaviour of the Prabhakar function for large arguments in the whole complex plane has been studied in \cite{GarraGarrappa2018_CNSNS,Paris2010,Paris2019}. In particular, for $0<\alpha\le 1$ it is
\[
	E_{\alpha,\beta}^{\gamma}(z) \sim \left\{\begin{array}{ll}
		{\mathcal F}_{\alpha,\beta}^{\gamma}(z) + {\mathcal A}_{\alpha,\beta}^{\gamma}(z \eu^{\mp\pi\iu})
		& |\arg z | < \frac{\alpha \pi}{2} \\
		{\mathcal A}_{\alpha,\beta}^{\gamma}(z \eu^{\mp\pi\iu}) + {\mathcal F}_{\alpha,\beta}^{\gamma}(z)
		& \frac{\alpha \pi}{2} < |\arg z | < \alpha \pi \\
		{\mathcal A}_{\alpha,\beta}^{\gamma}(z \eu^{\mp\pi\iu}) & \alpha \pi < |\arg z | \le \pi \\
	\end{array}\right. 
	\quad 
\]
as $|z| \to \infty$ and where the sign in $\eu^{\mp\pi\iu}$ must be understood as negative for $z$ in the upper complex half-plane and positive otherwise. We have adopted the convention proposed in \cite{GiustiColombaroGarraGarrappaPolitoPopolizioMainardi2020} by which in each sum it is first presented the dominant term. The exponential and algebraic expansions ${\mathcal F}_{\alpha,\beta}^{\gamma}(z)$ and ${\mathcal A}_{\alpha,\beta}^{\gamma}(z \eu^{\mp\pi\iu})$ are respectively 
\[
	{\mathcal F}_{\alpha,\beta}^{\gamma}(z) = \frac{1}{\Gamma(\gamma)} \eu^{z^{1/\alpha}} z^{\frac{\gamma-\beta}{\alpha}} \frac{1}{\alpha^{\gamma}} \sum_{k=0}^{\infty} c_{k} z^{-\frac{k}{\alpha}} 
\]
and
\[
	{\mathcal A}_{\alpha,\beta}^{\gamma}(z) = \frac{z^{-\gamma}}{\Gamma(\gamma)} \sum_{k=0}^{\infty} \frac{(-1)^{k} \Gamma(k+\gamma)}{k! \Gamma(\beta-\alpha(k+\gamma))} z^{-k} ,
\]
where $c_k$ are the coefficients in the inverse factorial expansion of 
\begin{equation}\label{eq:InverseFactorialExpansion}
	F_{\alpha,\beta}^{\gamma}(s) 
	= \frac{\Gamma(\gamma+s)\Gamma(\alpha s + 1-\gamma+\beta)}{\Gamma(s+1)\Gamma(\alpha s + \beta)} ,
\end{equation}
as $|s| \to \infty$, with $|\arg(s)| \le \pi - \epsilon$ for any arbitrarily small $\epsilon >0$. 
The first few entries of  coefficients $c_k$ are explicitly provided  in  \cite{Paris2019} but they can be evaluated by an algorithm described in \cite{Paris2010} and further explained in \cite{GarraGarrappa2018_CNSNS}.

For $\alpha,\beta,\gamma > 0$  the Prabhakar fractional integral of a function $f \in L_1[0,T]$ is the convolution of $f$ with the \textit{Prabhakar kernel} $e_{\alpha,\beta}^{\gamma}(t; \omega)$, namely
\begin{equation}\label{fractional.integral}
{}_0{\mathcal J}_{\alpha,\beta,\omega}^{\gamma} f(t) = \int_0^t  e_{\alpha,\beta}^{\gamma}(t-\tau;\omega) f(\tau) \du \tau .
\end{equation}
 
Its inverse operator regularized in Caputo's sense provides, in the case $0<\beta \le 1$ and for functions $f \in AC[0,T]$, the fractional Prabhakar derivative 
\begin{equation}\label{Prabhakar.operator}
    \DP^\gamma_{\alpha,\beta,\omega} f(t)=\int_0^t e^{-\gamma}_{\alpha,1-\beta}(t-\tau;\omega)f'(\tau)d\tau 
\end{equation}
(we refer to \cite{Giusti2020} for a discussion of the special case $\beta=1$).

\section{Asymptotic stability of linear systems of Prabhakar-type FDEs}\label{S:AsymptoticStability}

Due to the main interest in practical applications, we will assume throughout this paper that the parameters $\alpha$, $\beta$ and $\gamma$ fulfill the condition (\ref{cond.complete.monot}) under which the Prabhakar kernel is CM.

Consider the following linear system of Prabhakar-type fractional-order differential equations:
\begin{equation}\label{sys.Prabhakar}
     \DP^\gamma_{\alpha,\beta,\omega} y(t)=Ay(t),
\end{equation}
coupled with the initial condition $y(0)=y_0$, and where $\DP^\gamma_{\alpha,\beta,\lambda}$ is the  Prabhakar differential operator (regularized in the Caputo sense) defined according to (\ref{Prabhakar.operator}). 
 
System (\ref{sys.Prabhakar}) is equivalent to the following system of weakly singular Volterra integral equations of convolution type (see, for example \cite{Giusti2020,kochubei2011general}):\begin{equation}
    y(t) 
    = y_0 + 
    A \int_0^t  e_{\alpha,\beta}^{\gamma}(t-\tau;\omega) y(\tau) \du \tau.
\end{equation}


For the theory of linear Volterra integral equations, including the case when the convolution kernel is completely monotonic, we refer to   \cite{brunner2017volterra,gripenberg1990volterra,Lubich1986,tsalyuk1979volterra}.

From the LT of the Prabhakar kernel we observe that the characteristic equation associated to system \eqref{sys.Prabhakar} is 
\begin{equation}\label{eq.char}
\det\left(s^{\beta-\alpha\gamma}(s^\alpha-\omega )^\gamma I-A\right)=0 ,
\end{equation}
where, according to \cite{Doetsch}, the principal values (first branches) of the complex power functions are taken into account.

It is easy to see that $s$ is a root of the characteristic equation \eqref{eq.char} if and only if there exists an eigenvalue $\lambda$ of the matrix $A$ such that 
\begin{equation}\label{eq.char.lambda}
    s^{\beta-\alpha\gamma}(s^\alpha-\omega )^\gamma=\lambda.
\end{equation}

We obtain the following characterisation of the asymptotic stability of system \eqref{sys.Prabhakar}, in terms of the roots of its characteristic equation:
\begin{proposition}
The linear system \eqref{sys.Prabhakar} is asymptotically stable if and only if 
\[\sigma(A)\subset S_{\alpha,\beta,\omega}^\gamma\]
where $\sigma(A)$ denotes the spectrum of the matrix $A$ and
\[S_{\alpha,\beta,\omega}^\gamma=\{\lambda\in\mathbb{C}~:~ s^{\beta-\alpha\gamma}(s^\alpha-\omega )^\gamma\neq \lambda,~\forall ~\Re(s)\geq 0\}.\]
\end{proposition}

In what follows, we will give a complete characterisation of the \textit{stability region} $S_{\alpha,\beta,\omega}^\gamma$.

\section{Stability region by the root locus method}\label{S:StabilityRegion}

The boundary of the \textit{stability region} $S_{\alpha,\beta,\omega}^\gamma$ will be determined using the root locus method. We first give the following preliminary results

\begin{lemma}\label{lem.Lambda} The function $\Lambda:\left[0,\frac{\alpha\pi}{2}\right)\rightarrow \mathbb{C}$ given by:
\[\Lambda(\theta)=|\omega|^{\frac{\beta}{\alpha}}~\dfrac{(\sin\theta)^{\frac{\beta}{\alpha}-\gamma}\left(\sin\frac{\alpha\pi}{2}\right)^\gamma}{\left(\sin\left(\frac{\alpha\pi}{2}-\theta\right)\right)^\frac{\beta}{\alpha}}~ \eu^{\iu\left[\gamma\theta+(\beta-\alpha\gamma)\frac{\pi}{2}\right]}\]
is a $C^\infty$ function which satisfies 
\[\lim_{\theta\rightarrow 0} \Lambda(\theta)=\begin{cases}
0 &,~\text{if}~\beta-\alpha\gamma>0,\\
|\omega|^\gamma &,~\text{if}~\beta-\alpha\gamma=0.
\end{cases}\]
\[\lim_{\theta\rightarrow\frac{\alpha\pi}{2}}|\Lambda(\theta)|=\infty\quad\text{and}\quad \lim_{\theta\rightarrow\frac{\alpha\pi}{2}}\Arg(\Lambda(\theta))=\frac{\beta\pi}{2}.\]
Moreover: 
\[0\leq \Arg(\Lambda(\theta))\leq \frac{\beta\pi}{2}, \quad \forall~ \theta\in\left[0,\frac{\alpha\pi}{2}\right).\]
The image of the function $\Lambda$ in the complex plane, i.e. the curve $\Psi_{\alpha,\beta,\omega}^\gamma$ defined by the parametric equation
\[\Psi_{\alpha,\beta,\omega}^\gamma:
\quad \lambda=\Lambda(\theta) ,\quad\theta\in \left[0,\frac{\alpha\pi}{2}\right), 
\]
is a simple curve, included in the first quadrant of the complex plane. 
\end{lemma}

\begin{proof}
The first part of the proof is trivial. Moreover, as inequalities \eqref{cond.complete.monot} hold and $0\leq \theta<\frac{\alpha\pi}{2}$, we have:
\[0\leq\frac{ (\beta-\alpha\gamma)\pi}{2}<\Arg(\Lambda(\theta))<\frac{\beta\pi}{2}\leq \frac{\pi}{2}.\]
Therefore, the curve $\Psi_{\alpha,\beta,\omega}^\gamma$ is included in the first quadrant of the complex plane. 

Assuming by contradiction that $\Psi_{\alpha,\beta,\omega}^\gamma$  is not simple, there exist $\theta$ and $\theta'$ such that $0\leq\theta<\theta'<\frac{\alpha\pi}{2}$ and $\Lambda(\theta)=\Lambda(\theta')$. Therefore, $|\Lambda(\theta)|=|\Lambda(\theta')|$, or equivalently: 
\[\frac{\sin(\frac{\alpha\pi}{2}-\theta)}{\sin(\frac{\alpha\pi}{2}-\theta')}=\left(\frac{\sin\theta}{\sin \theta'}\right)^{1-\frac{\alpha\gamma}{\beta}}.\]
As $0\leq\theta<\theta'<\frac{\alpha\pi}{2}$, the left hand side of this equality is larger than $1$, while the right hand side is subunitary, which is absurd. Hence, $\Psi_{\alpha,\beta,\omega}^\gamma$ is a simple curve.
\qed\end{proof}

Clearly, when $\beta-\alpha\gamma=0$, the  parametric equation of the curve given by Lemma \ref{lem.Lambda}  simplifies to
\[\Psi_{\alpha,\alpha\gamma,\omega}^\gamma:
\quad \lambda=\left(\dfrac{|\omega|\sin\frac{\alpha\pi}{2}}{\sin\left(\frac{\alpha\pi}{2}-\theta\right)}\right)^\gamma \! \eu^{\iu\gamma\theta},
\quad \theta\in \left[0,\frac{\alpha\pi}{2}\right).
\]

In what follows, $\overline{\Psi}_{\alpha,\beta,\omega}^\gamma$ denotes the complex conjugate of the curve $\Psi_{\alpha,\beta,\omega}^\gamma$ defined in Lemma \ref{lem.Lambda}, i.e.
\[\overline{\Psi}_{\alpha,\beta,\omega}^\gamma=\{\lambda\in\mathbb{C}~:~\overline{\lambda}\in \Psi_{\alpha,\beta,\omega}^\gamma\}.\]

We obtain the following result, characterising the root locus of the characteristic equation \eqref{eq.char.lambda}:

\begin{proposition}\label{prop.root.locus}
The characteristic equation \eqref{eq.char.lambda} has pure imaginary roots if and only if $\lambda\in\Psi_{\alpha,\beta,\omega}^\gamma\cup \overline{\Psi}_{\alpha,\beta,\omega}^\gamma$.
\end{proposition}
    
\begin{proof}
Assuming that the equation \eqref{eq.char.lambda} has a root $s=i\mu$, with $\mu\geq 0$, let us consider $\rho>0$ and $\theta\in(-\pi,\pi]$ such that 
\[(\iu \mu)^\alpha-\omega=\rho \eu^{\iu\theta}.\]
Hence, equation \eqref{eq.char.lambda} has a pure imaginary root if and only if there exist $\mu\geq 0$ and $\rho>0$ and $\theta\in(-\pi,\pi]$ such that \begin{equation}\label{sys.pure.imaginary}
\begin{cases}
(\iu \mu)^{\beta-\alpha\gamma}(\rho \eu^{\iu\theta})^\gamma=\lambda\\
(\iu\mu)^\alpha-\omega=\rho \eu^{\iu\theta}
\end{cases}    
\end{equation}
Taking the real and imaginary parts in the second equation of system \eqref{sys.pure.imaginary}, it follows that 
\[
\begin{cases}
\mu^\alpha \cos\frac{\alpha\pi}{2}=\rho\cos\theta+\omega\\
\mu^\alpha \sin\frac{\alpha\pi}{2}=\rho\sin\theta
\end{cases}    
\]
and hence:
\begin{equation}\label{eq.mu.rho}
\mu^\alpha = \dfrac{|\omega|\sin\theta}{\sin\left(\frac{\alpha\pi}{2}-\theta\right)}\quad\text{and}\quad
\rho =\dfrac{|\omega|\sin\frac{\alpha\pi}{2}}{\sin\left(\frac{\alpha\pi}{2}-\theta\right)}.
\end{equation}
It is obvious that since $\mu\geq 0$ and $\rho>0$, the following inequalities must be satisfied:
\[\sin\theta\geq 0\quad\text{and}\quad \sin\left(\frac{\alpha\pi}{2}-\theta\right)>0.\]
which is equivalent to $\theta\in \left[0,\frac{\alpha\pi}{2}\right)$.   

Replacing $\mu$ and $\rho$ given by \eqref{eq.mu.rho} into the first equation of system \eqref{sys.pure.imaginary}, we deduce that $\lambda\in \Psi_{\alpha,\beta,\omega}^\gamma$. In a similar way, assuming that equation \eqref{eq.char.lambda} has a root $s=-\iu \mu$, with $\mu\geq 0$, it follows that $\lambda\in \overline{\Psi}_{\alpha,\beta,\omega}^\gamma$. 
\qed\end{proof}

Let us denote by $N(\alpha,\beta,\gamma,\omega,\lambda)$ the number of unstable roots ($\Re(s)\geq 0$) of the characteristic equation \eqref{eq.char.lambda}, including their multiplicities. The following lemma shows that the function $N(\alpha,\beta,\gamma,\omega,\lambda)$ is well-defined. Moreover, some important properties are also established, which are needed for the proof of the main results. 

\begin{lemma}\label{lemma.number.roots} Let $\lambda\in\mathbb{C}$ and $\alpha,\beta,\gamma,\omega$ satisfy inequalities \eqref{cond.complete.monot}. The following statements hold:
\begin{itemize}
\item[i.] The characteristic function equation \eqref{eq.char.lambda} has at most a finite number of roots satisfying $\Re(s)\geq 0$. 

\item[ii.]  The function $\lambda\mapsto N(\alpha,\beta,\gamma,\omega,\lambda)$ is continuous at each $\lambda\notin \Psi_{\alpha,\beta,\omega}^\gamma \cup \overline{\Psi}_{\alpha,\beta,\omega}^\gamma$, and hence, $N(\alpha,\beta,\gamma,\omega,\lambda)$ is constant on each connected component of the set $\mathbb{C}\setminus (\Psi_{\alpha,\beta,\omega}^\gamma \cup \overline{\Psi}_{\alpha,\beta,\omega}^\gamma)$.
\end{itemize}
\end{lemma}

\begin{proof} Let us denote \[\Delta(s;\alpha,\beta,\gamma,\omega,\lambda)=s^{\beta-\alpha\gamma}(s^\alpha-\omega)^\gamma-\lambda.\] 

We will first show that the set of unstable roots of the equation \eqref{eq.char.lambda} is bounded. Indeed, if $s$ is a root of \eqref{eq.char.lambda}  such that $\Re(s)\geq 0$, as $\alpha\in(0,1]$, it follows that $\Re(s^\alpha)\geq 0$. Moreover, as $\omega<0$, we have: \[|s^\alpha-\omega|=\sqrt{|s|^{2\alpha}+\omega^2-2\omega\Re(s^\alpha)}\geq |s|^{\alpha}.\]
Therefore:
\[|\lambda|=|s|^{\beta-\alpha\gamma}|s^\alpha-\omega|^\gamma\geq |s|^{\beta-\alpha\gamma}|s|^{\alpha\gamma}=|s|^\beta,\]
and therefore, $|s|\leq |\lambda|^{\frac{1}{\beta}}$.

\noindent\textit{Proof of statement (i).} 
Let us first consider $\lambda\neq 0$. Assuming that the characteristic equation \eqref{eq.char.lambda} has an infinite number of unstable roots, the Bolzano-Weierstrass theorem implies that there exists a convergent sequence of unstable roots $(s_n)$ with the limit $s_0\neq 0$, such that $\Re(s_0)\geq 0$. Since the function $\Delta(s;\alpha,\beta,\gamma,\omega,\lambda)$ is analytic in $\mathbb{C}\setminus\mathbb{R}_-$, the principle of permanence implies that it is identically zero, which is absurd. Hence, the function $N(\alpha,\beta,\gamma,\omega,\lambda)$ is finite and well-defined. 

If $\lambda=0$, the number of unstable roots of \eqref{sys.Prabhakar} is finite  because the equation $s^\alpha-\omega=0$ has a finite number of unstable roots. This can be shown in a similar way as above, by a simple application of the principle of permanence. 

\noindent\textit{Proof of statement (ii).}  Let $\lambda_0\in \mathbb{C}\setminus (\Psi_{\alpha,\beta,\omega}^\gamma \cup \overline{\Psi}_{\alpha,\beta,\omega}^\gamma)$ and $r>0$ such that the open neighborhood $B_r(\lambda_0)=\{\lambda\in\mathbb{C}~:~|\lambda-\lambda_0|<r\}$ is included in the set $\mathbb{C}\setminus (\Psi_{\alpha,\beta,\omega}^\gamma \cup \overline{\Psi}_{\alpha,\beta,\omega}^\gamma)$.

For any $\lambda\in B_r(\lambda_0)$ we have that  
$|\lambda|<r+|\lambda_0|$, and hence, based on the first part of the proof, any unstable root of $\Delta(s;\alpha,\beta,\gamma,\omega,\lambda)$ satisfies: 
\[|s|< (r+|\lambda_0|)^{\frac{1}{\beta}}.\]

Let us denote by $(c)$  the simple closed curve, oriented counterclockwise, bounding the open half-disk
$$D=\{s\in\mathbb{C}~:~\Re(s)> 0,~ 0<|s|<(r+|\lambda_0|)^{\frac{1}{\beta}}\}.$$
By the above construction and Proposition \ref{prop.root.locus}, it is clear that for any $\lambda\in B_r(\lambda_0)$, all unstable roots of the characteristic function $\Delta(s;\alpha,\beta,\gamma,\omega,\lambda)$ belong to $D$.

As $\Delta(s;\alpha,\beta,\gamma,\omega,\lambda_0)\neq 0$ for any $s\in(c)$, it is easy to see that \[m_0=\min\limits_{s\in(c)}|\Delta(s;\alpha,\beta,\gamma,\omega,\lambda_0)|>0.\]

Considering $r'=\min\{m_0,r\}$ it follows that for any $s\in (c)$ and for any $\lambda\in B_{r'}(\lambda_0)\subset B_r(\lambda_0)$, we have:
\begin{align*}|&\Delta(s;\alpha,\beta,\gamma,\omega,\lambda)-\Delta(s;\alpha,\beta,\gamma,\omega,\lambda_0)|=\\
&=|\lambda-\lambda_0|<r'\leq m_0\leq |\Delta(s;\alpha,\beta,\gamma,\omega,\lambda_0)|.
\end{align*}
By Rouch\'{e}'s theorem, it follows that  $\Delta(s;\alpha,\beta,\gamma,\omega,\lambda_0)$ and $\Delta(s;\alpha,\beta,\gamma,\omega,\lambda)$ have the same number of zeros in the half-disk $D$, 
and hence $$N(\alpha,\beta,\gamma,\omega,\lambda)=N(\alpha,\beta,\gamma,\omega,\lambda_0)\quad,~\forall~ \lambda\in B_{r'}(\lambda_0).$$ Hence, the function
$\lambda\mapsto N(\alpha,\beta,\gamma,\omega,\lambda)$ is continuous on $ \mathbb{C}\setminus (\Psi_{\alpha,\beta,\omega}^\gamma \cup \overline{\Psi}_{\alpha,\beta,\omega}^\gamma)$, and from the fact that it is integer-valued, we deduce that it is constant on each connected component of $ \mathbb{C}\setminus (\Psi_{\alpha,\beta,\omega}^\gamma \cup \overline{\Psi}_{\alpha,\beta,\omega}^\gamma)$.
\qed\end{proof}

We now give the main result which characterises the stability region $S_{\alpha,\beta,\omega}^\gamma$ of system \eqref{sys.Prabhakar}.

\begin{theorem}\label{thm.stab.region}
The stability region $S_{\alpha,\beta,\omega}^\gamma$ of system \eqref{sys.Prabhakar} is the region of the complex plane which includes $\mathbb{C}_-=\{\lambda\in\mathbb{C}~:\Re(\lambda)<0\}$ and is bounded by $\Psi_{\alpha,\beta,\omega}^\gamma$ and its complex conjugate $\overline{\Psi}_{\alpha,\beta,\omega}^\gamma$.
\end{theorem}

\begin{proof}
Lemma \ref{lem.Lambda} implies that $\Psi_{\alpha,\beta,\omega}^\gamma \cup \overline{\Psi}_{\alpha,\beta,\omega}^\gamma$ partition the complex plane into two disjoint regions, which will be denoted by $D_-$ and $D_+$. As $\Psi_{\alpha,\beta,\omega}^\gamma$ and $\overline{\Psi}_{\alpha,\beta,\omega}^\gamma$ are included, respectively, in the first and fourth quadrant of complex plane, one of these regions includes $\mathbb{C}_-$ (we will further assume that $\mathbb{C}_-\subset D_-$).  Moreover, based on Lemma \ref{lemma.number.roots},  these regions have the property that, for every $\lambda$ within a given region, the number of unstable roots of the characteristic equation \eqref{eq.char.lambda} is constant.

In what follows, we will show that if $\lambda\in(-\infty,0)$, the characteristic equation \eqref{eq.char.lambda} does not have any roots with positive real part. Indeed, let us assume by contradiction that there exists $s\in\mathbb{C}$, $\Re(s)\geq 0$ such that 
\[s^{\beta-\alpha\gamma}(s^\alpha-\omega)^\gamma=\lambda.\]
As both $s$ and $\overline{s}$ are roots of the above equation, we may further assume that $\Arg(s)\in\left[0,\frac{\pi}{2}\right]$. 

On one hand, we have: 
\begin{equation}\label{eq.arg}
    \Arg\left[s^{\beta-\alpha\gamma}(s^\alpha-\omega)^\gamma\right]=\Arg(\lambda)=\pi.
\end{equation}

On the other hand, as $\beta-\alpha\gamma\in[0,1)$, we have:
\begin{align*}\Arg(s^{\beta-\alpha\gamma})\!&=(\beta-\alpha\gamma)\Arg(s)+2\pi\left\lfloor \frac{\pi-(\beta-\alpha\gamma)\Arg(s)}{2\pi} \right\rfloor\\
&=(\beta-\alpha\gamma)\Arg(s)\in \left[0,\frac{\pi}{2}\right].\end{align*}
Moreover, as $\omega<0$ and $\alpha\in(0,1]$, we deduce that $\Arg(s^\alpha)\in\left[0,\frac{\pi}{2}\right]$ and:
\[0<\Arg(s^\alpha-\omega)<\Arg(s^\alpha)=\alpha\Arg(s)\leq\frac{\pi}{2}\]
and hence, as $0<\alpha\gamma\leq 1$, it follows that: 
\[0<\gamma\Arg(s^\alpha-\omega)<\alpha\gamma \Arg(s)\leq \frac{\pi}{2}.\]
Therefore: 
\begin{align*}\Arg\left((s^\alpha-\omega)^\gamma\right)&=\gamma\Arg(s^\alpha-\omega)+2\pi\left\lfloor \frac{\pi-\gamma\Arg(s^\alpha-\omega)}{2\pi} \right\rfloor\\
&=\gamma\Arg(s^\alpha-\omega)\in \left[0,\frac{\pi}{2}\right].\end{align*}
Finally, combining the previous results and taking into account that $\beta\leq 1$, we get:
\begin{align*}
0\leq \Arg&\left[s^{\beta-\alpha\gamma}(s^\alpha-\omega)^\gamma\right]\\
&=\Arg(s^{\beta-\alpha\gamma})+ \Arg\left((s^\alpha-\omega)^\gamma\right)\\
&=(\beta-\alpha\gamma)\Arg(s)+\gamma\Arg(s^\alpha-\omega)\\
&< (\beta-\alpha\gamma)\Arg(s)+\alpha\gamma\Arg(s)\\
&=\beta\Arg(s)\\
&\leq \frac{\pi}{2}.
\end{align*}
which is in contradiction with \eqref{eq.arg}. Therefore, all the roots of the characteristic equation \eqref{eq.char.lambda} are in the left half plane, whenever $\lambda\in(-\infty,0)$. Hence, based on Lemma \ref{lemma.number.roots}, we obtain that $N(\alpha,\beta,\gamma,\omega,\lambda)=0$ for any $\lambda\in D_-$,  implying that $D_-\subset S_{\alpha,\beta,\omega}^\gamma$.

\noindent\textit{Proof of the transversality condition.}

Let us denote by $s(\lambda)$ the unique root of the characteristic equation \eqref{eq.char.lambda} such that $\Re\left[s(\lambda^\star)\right]=0$ when $\lambda^\star\in \Psi_{\alpha,\beta,\omega}^\gamma$. Let $\theta\in\left(0,\frac{\alpha\pi}{2}\right)$ such that $\lambda^\star=\Lambda(\theta)$, and hence, based on the proof of Proposition \ref{prop.root.locus}, we have: 
\[s(\lambda^\star)=\iu\mu= \iu\left(\dfrac{|\omega|\sin\theta}{\sin\left(\frac{\alpha\pi}{2}-\theta\right)}\right)^{\frac{1}{\alpha}}.\]
Let us denote $F(s)=s^{\beta-\alpha\gamma}(s^\alpha-\omega)^\gamma$. By the implicit function theorem, as $F(s(\lambda))=\lambda$,  we have: 
\[\frac{\partial s}{\partial \Re(\lambda)}=\frac{1}{F'(s)}\quad\text{and}\quad\frac{\partial s}{\partial \Im(\lambda)}=\frac{\iu}{F'(s)}\]
and hence:
\[\frac{\partial \Re(s)}{\partial \Re(\lambda)}=\frac{\Re(F'(s))}{|F'(s)|^2}\quad\text{and}\quad\frac{\partial \Re(s)}{\partial \Im(\lambda)}=\frac{\Im(F'(s))}{|F'(s)|^2},\]
which leads to:
\[\nabla\Re(s):=\frac{\partial \Re(s)}{\partial \Re(\lambda)}+ \iu \frac{\partial \Re(s)}{\partial \Im(\lambda)}=\frac{F'(s)}{|F'(s)|^2}.\]
A simple computation shows that: 
\[F'(s)=s^{-1}\left[\beta+\alpha\gamma\omega (s^\alpha-\omega)^{-1}\right]F(s),\]
which gives:
\[
    \begin{aligned}
F'(s)\vert_{\lambda=\lambda^\star}&=F'(\iu\mu) \\
&=-\iu \mu^{-1}\left[\beta+\alpha\gamma\omega ((i\mu)^\alpha-\omega)^{-1}\right]\lambda^\star. \        
    \end{aligned}
\]
Finally, taking into account that $\lambda^\star=\Lambda(\theta)$ and expressing $(\iu \mu)^\alpha-\omega$ in terms of $\theta$, it follows that: 
\[F'(s)\vert_{\lambda=\lambda^\star}=-\iu \mu^{-1}\left[\beta-\alpha\gamma \dfrac{\sin\left(\frac{\alpha\pi}{2}-\theta\right)}{\sin\frac{\alpha\pi}{2}}\eu^{-\iu\theta}\right]\Lambda(\theta).\] 
On the other hand, a straightforward computation gives: 
\[\Lambda'(\theta)\!=\!\dfrac{\sin\frac{\alpha\pi}{2}}{\alpha\sin\theta\sin\left(\frac{\alpha\pi}{2}-\theta\right)}\!\left[\beta\!-\!\alpha\gamma \dfrac{\sin\left(\frac{\alpha\pi}{2}-\theta\right)}{\sin\frac{\alpha\pi}{2}}\eu^{-\iu\theta}\!\right]\!\Lambda(\theta)\]
and therefore:
\[F'(s)\vert_{\lambda=\lambda^\star}=-\iu \mu^{-1}\dfrac{\alpha\sin\theta\sin\left(\frac{\alpha\pi}{2}-\theta\right)}{\sin\frac{\alpha\pi}{2}}\Lambda'(\theta).\] 

Hence, exploiting the $\mathbb{R}^2$
vector space structure which underlies $\mathbb{C}$ and considering that the parametrization of the curve $\Psi_{\alpha,\beta,\omega}^\gamma$ is fixed in the direction of increasing $\theta$, it follows that the gradient vector $\nabla \Re(s)(\lambda^\star)$ is in fact a right-pointing normal vector to the curve  $\Psi_{\alpha,\beta,\omega}^\gamma$, pointing towards the region $D_+$. We deduce that as the parameter $\lambda$ crosses the curve $\Psi_{\alpha,\beta,\omega}^\gamma$ from the region $D_-$ into the region $D_+$, $\Re(s(\lambda))$ becomes positive, which ensures that the transversality condition holds. 

Moreover, this shows that if $\lambda\in D_+$, the characteristic equation \eqref{eq.char.lambda} has at least one root with positive real part, and hence, we finally obtain that $S_{\alpha,\beta,\omega}^\gamma=D_-$, which completes the proof.
\qed\end{proof}

It is important to emphasize that Theorem \ref{thm.stab.region} gives a complete characterisation of the stability region $S_{\alpha,\beta,\omega}^\gamma$ of system \eqref{sys.Prabhakar}. Some examples are shown in Figures  \ref{fig:SR.increasing.gamma} and \ref{fig:SR.decreasing.omega}. 

\begin{remark}
Based on Lemma \ref{lem.Lambda}, it is clear that when $\gamma\rightarrow 0$ (i.e. when the Prabhakar derivative in \eqref{sys.Prabhakar} reduces to the standard Caputo derivative of order $\beta$), the curve $\Psi^\gamma_{\alpha,\beta,\omega}$ approaches the half line $\Arg(\lambda)=\frac{\beta\pi}{2}$ of the complex plane, and hence the stability region is indeed 
\[S_{\alpha,\beta,\omega}^0=\{\lambda\in\mathbb{C}~:~|\Arg(\lambda)|>\frac{\beta\pi}{2}\},\]
which is in accordance with Matignon's theorem \cite{Matignon}. Hence, Theorem \ref{thm.stab.region} is a generalization of Matignon's theorem for the case of systems of fractional differential equations with Prabhakar derivatives. 
\end{remark}

\begin{remark}
The transversality condition which was verified in the proof of Theorem \ref{thm.stab.region} ensures that if $\lambda\in\mathbb{C}\setminus\{0\}$ is a simple eigenvalue of the matrix $A$ of the linear system \eqref{sys.Prabhakar}, when the parameters $(\alpha,\beta,\gamma,\omega)$ of the Prabhakar derivative are varied and $\lambda$ crosses from $S_{\alpha,\beta,\omega}^\gamma$ to its open complementary $\text{int}(\mathbb{C}\setminus S_{\alpha,\beta,\omega}^\gamma)$, exactly one root of the characteristic equation \eqref{eq.char} crosses the imaginary axis from the left half-plane to the right half-plane of $\mathbb{C}$. More generally, based on a similar argument, we can express the number of unstable roots ($\Re(s)\geq 0$) of the characteristic equation \eqref{eq.char} as follows: 
\[
N(\alpha,\beta,\gamma,\omega)=\!\!\sum_{\lambda\in\sigma(A)}\!\!N(\alpha,\beta,\gamma,\omega,\lambda)=\!\!\!\!\sum_{\lambda\in \mathbb{C}\setminus S_{\alpha,\beta,\omega}^\gamma} \!\!\!\! m(\lambda),
\]
where $m(\lambda)$ denotes the algebraic multiplicity of the eigenvalue $\lambda$. 
\end{remark}

\begin{figure*}[htbp]
    \centering
    \includegraphics[width=\linewidth]{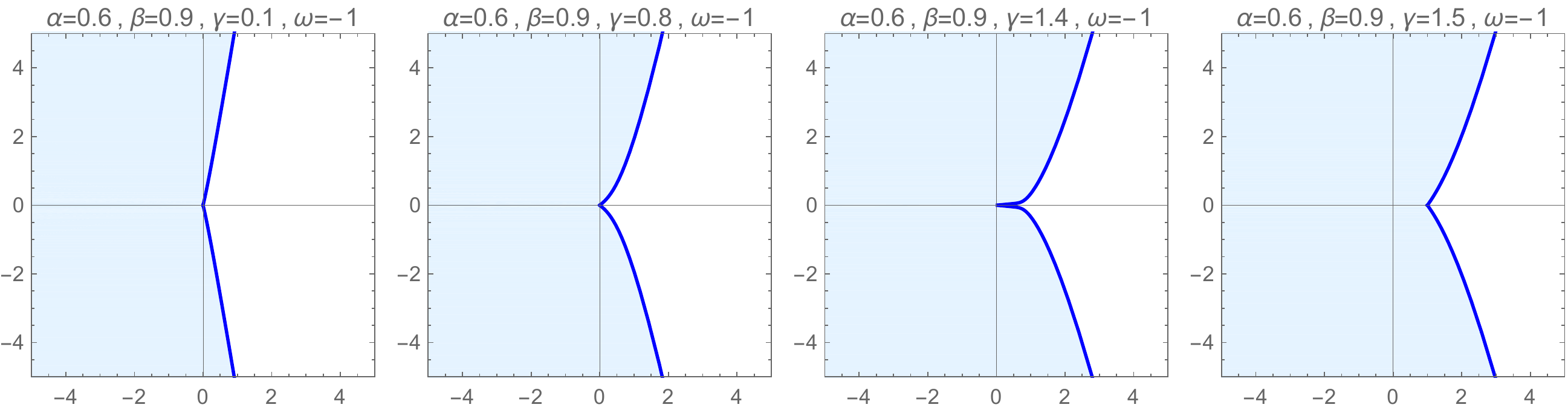}
    \caption{Stability region $S_{\alpha,\beta,\omega}^\gamma$ for fixed values of $\alpha, \beta, \omega$ and increasing values of $\gamma$. Last figure is for the special case $\beta-\alpha\gamma=0$.}
    \label{fig:SR.increasing.gamma}
\end{figure*}

\begin{figure*}[htbp]
    \centering
    \includegraphics[width=\linewidth]{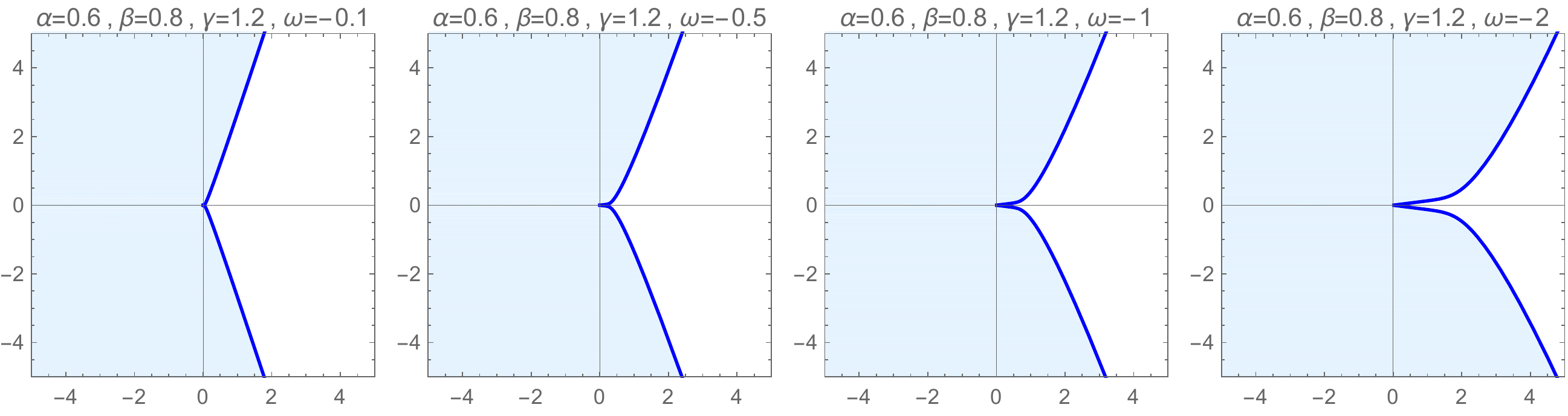}
    \caption{Stability region $S_{\alpha,\beta,\omega}^\gamma$ for fixed values of $\alpha, \beta, \gamma$ and decreasing values of $\omega$.}
    \label{fig:SR.decreasing.omega}
\end{figure*}

\section{Solution of linear and nonlinear FDEs of Prabhakar type}\label{S:Solution}

Analytical solutions of (\ref{sys.Prabhakar}) are not available in a simple and closed form. With the purpose of verifying the theoretical findings on  stability regions, we derive here asymptotic representations of the exact solution  for small and large arguments, together with a numerical method for solving more general nonlinear problems with the fractional Prabhakar derivative. 

For convenience we derive asymptotic expansion just for the scalar case $y(t):[0,T]\to \Cset$ and $A\in \Cset$; the generalization to the vector case is however straightforward. 

\subsection{Asymptotic expansion for small arguments}

By means of the LT we can rewrite (\ref{sys.Prabhakar}) in the LT domain as
\[
    s^{\beta-\alpha\gamma-1}\bigl(s^{\alpha}-\omega\bigr)^{\gamma} \bigl(s \hat{y}(s)  - y_0 \bigr) = A \hat{y}(s) ,
\]  
with $\hat{y}(s)$ the LT of $y(t)$. Therefore the solution in the LT domain is $\hat{y}(s) = {\cal H}(s) y_0$, where
\begin{equation}\label{eq:LinSis_Sol_LT}
    {\cal H}(s) = \frac{s^{\beta-\alpha\gamma-1}(s^{\alpha}-\omega)^{\gamma} }{s^{\beta-\alpha\gamma}(s^{\alpha}-\omega)^{\gamma} - A}.
\end{equation}

Observe now that 
\[
    {\cal H}(s) = \displaystyle\frac{s^{-1} \Bigl(1-\displaystyle\frac{\omega}{s^{\alpha}}\Bigr)^{\gamma} }{\Bigl(1-\displaystyle\frac{\omega}{s^{\alpha}}\Bigr)^{\gamma} - \frac{A}{s^{\beta}} } 
    = s^{-1} \left( 1- \frac{A}{s^{\beta}\Bigl(1-\frac{\omega}{s^{\alpha}}\Bigr)^{\gamma}}\right)^{-1}  
\]
and, for sufficiently large $|s|$, we can expand
\[
{\cal H}(s) 
= s^{-1}\sum_{j=0}^{\infty} \frac{A^j}{s^{j \beta}\Bigl(1-\frac{\omega}{s^{\alpha}}\Bigr)^{j \gamma}}  
=  \sum_{j=0}^{\infty} \frac{s^{\alpha \gamma j - j \beta - 1}  A^j}{(s^{\alpha}-\omega)^{j\gamma}} .
\]

Therefore, after inverting back each  LT of Prabhakar functions in the series we are able to obtain
\[
 y(t) = \sum_{j=0}^{\infty} A^j t^{j \beta} E_{\alpha, j \beta +1}^{j \gamma} (\omega t^{\alpha}) y_0
\]
which holds as $t \to 0$.

\subsection{Asymptotic expansion for large arguments}

To derive an asymptotic expansion of the solution $y(t)$ of (\ref{sys.Prabhakar}) as $t\to\infty$  we consider again the function ${\mathcal H}(s)$.

Since we are now interested in study the solution $\hat{y} = {\mathcal H}(s)y_0$ in the LT domain as $|s|\to 0$, we have to take into account the singularity of ${\mathcal H}(s)$.

Due to the
transversality condition stated by Theorem \ref{thm.stab.region}, ${\mathcal H}(s)$ has just one singularity, say $\bar{s}$, which can be eliminated in the formula for the inversion of the LT by the residue subtraction
\[
    h(t) = \mathop{Res} \bigl( \eu^{st} {\cal H}(s) , \bar{s} \bigr) + \hat{h}(t) ,
\]
where 
\[
\hat{h}(t) =
		\frac{1}{2\pi i} \int_{\mathcal C} \eu^{st} {\cal H}(s) \du s 
\]
and  $\mathcal C$ is any contour in the complex plane leaving $\bar{s}$ at its right and not crossing the branch-cut placed on the negative real semi-axis.

An analytical expression of $\bar{s}$ seems not available and therefore $\bar{s}$ must be evaluated numerically after solving the equation $s^{\beta-\alpha\gamma}(s^{\alpha}-\omega)^{\gamma} = A$ or, equivalently,
\[
    s^{\mu} - \omega s^{\mu-\alpha} - B = 0
    , \quad \mu = \frac{\beta}{\gamma}, \quad B = A^{\frac{1}{\gamma}} .
\]

The corresponding residue can be instead evaluated by simple derivations. Indeed, since we assume $0<\alpha<1$ and real $\omega <0$, it is $\bar{s}^{\alpha}-\omega \not=0$; therefore, after standard derivations one obtains
\[
\mathop{Res} \bigl( \eu^{st} {\cal H}(s) , \bar{s} \bigr) = C^{\gamma}_{\alpha,\beta,\omega}(\bar{s}) \eu^{\bar{s}t} ,
\]
where $C^{\gamma}_{\alpha,\beta,\omega}(s)$ is constant with respect to $t$ and
\[
C^{\gamma}_{\alpha,\beta,\omega}(s) = \frac{ (s^{\alpha} - \omega)}{\beta s^{\alpha} - (\beta-\alpha \gamma) \omega}.
\]

To evaluate $\hat{h}(t)$ we first observe from (\ref{eq:LinSis_Sol_LT}) that 
\[
    {\cal H}(s)
    = - s^{\beta-\alpha\gamma-1} \frac{(s^{\alpha}-\omega)^{\gamma}}{A}
    \left(1 - s^{\beta-\alpha\gamma} \frac{(s^{\alpha}-\omega)^{\gamma}}{A} \right)^{-1} , 
\]
and hence, for  sufficiently small $|s|$, it is  possible to consider the expansion
\[
\begin{aligned} 
    {\cal H}(s)
    &=- s^{\beta-\alpha\gamma-1} \frac{(s^{\alpha}-\omega)^{\gamma}}{A} \sum_{k=0}^{\infty} s^{k\beta-k\alpha\gamma}\frac{(s^{\alpha}-\omega)^{k \gamma}}{A^k}  \\
    &= - \sum_{k=0}^{\infty} s^{(k+1)\beta-(k+1)\alpha\gamma-1 }\frac{(s^{\alpha}-\omega)^{(k+1) \gamma}}{A^{(k+1)}}  \\
    &= - \sum_{k=1}^{\infty} s^{k\beta-k\alpha\gamma-1}\frac{(s^{\alpha}-\omega)^{k \gamma}}{A^{k}} 
\end{aligned}
\]

We can therefore transform back ${\cal H}(s)$ from the LT domain to the time domain to obtain $\hat{h}(t)$ and hence the expansion of the solution $y(t)$ as $t \to \infty$
\[
    y(t) = \left[ C^{\gamma}_{\alpha,\beta,\omega}(\bar{s}) \eu^{\bar{s}t}  - \sum_{k=1}^{\infty} \frac{t^{-k\beta}}{A^k}  E_{\alpha,1-k\beta}^{-k\gamma}(t^{\alpha} \omega) \right] y_0 .        
\]

This formula  can be exploited, in connection with the asymptotic representation of the Prabhakar function introduced in Section \ref{S:Background} to provide a representation of the solution $y(t)$ of (\ref{sys.Prabhakar}).

\subsection{Numerical solution}

To devise an effective method for solving not only the linear system (\ref{sys.Prabhakar}) but,  more generally, any nonlinear system such as
\begin{equation}\label{eq:PrabSysNonLinear}
    \left\{\begin{array}{l}
     \DP^\gamma_{\alpha,\beta,\omega} y(t) = f(t,y(t)) \\
     y(0) = y_0 \
     \end{array}\right. ,
\end{equation}
we use as starting point the standard trapezoidal rule
\[
y_{n+1} - y_{n} = \frac{h}{2} \bigl(f(t_n,y_n) + f(t_{n+1},y_{n+1})\bigr)
\]
for ordinary differential equations and its generalization to our problem is made in the framework devised by Lubich \cite{Lubich1988a,Lubich1988b,Lubich2004}. The choice of the trapezoidal rule as starting point for devising a numerical method  for the solution of  (\ref{eq:PrabSysNonLinear}) is motivated by its excellent stability properties. Given the generating function of the trapezoidal rule
\[
    \delta(\xi) = \frac{2(1-\xi)}{1+\xi} ,
\]
a corresponding trapezoidal convolution quadrature rule for (\ref{eq:PrabSysNonLinear}) evaluates the approximation $y_n$ of $y(t_n)$ by the formula
\[
    y_n = y_0 + h^{\beta}\sum_{j=0}^{s} w_{n,j}  f(t_j,y_j) + h^{\beta} \sum_{j=0}^{n} c_{n-j}  f(t_j,y_j) .
\]

Convolution weights $c_n$ are the coefficients in the asymptotic expansion of
\[
    \sum_{n=0}^{\infty} c_n \xi^n = \frac{1}{h^{\beta}} G(\xi)
    , \quad
    G(\xi) = {\mathcal E}_{\alpha,\beta} \Bigl( \frac{\delta(\xi)}{h} ; \omega \Bigr) ,
\]
with ${\mathcal E}_{\alpha,\beta} (s ; \omega) $ the LT of ${\mathcal E}_{\alpha,\beta} (s ; \omega) $,
and can be evaluated with high accuracy by a quadrature rule applied to the Cauchy integral 
\[
c_n = \frac{1}{2 \pi \iu} \int_{\mathcal C} \xi^{-n-1} G(\xi) \du \xi 
\]
with ${\mathcal C}$ a suitably selected closed contour encircling the origin but not any singularity of $G(\xi)$. Starting weights $w_{n,j}, j=0,1,\dots,s$ are instead introduced to deal with the lack of smoothness at $0$ of the solution and evaluated after imposing that exact solutions are obtained when $f(t,y(t))=t^{\nu}$, with $\nu$ multiple of $\beta$ less than 1.  We refer again to \cite{Lubich1988a,Lubich1988b,Lubich2004} for a more detailed description.

\section{Numerical experiments}\label{S:NumericalExperiments}

\subsection{Asymptotic stability}

To verify the theoretical findings on the asymptotic stability we consider here, for the selection of the parameters $\alpha=0.8$, $\beta=0.9$, $\gamma=0.8$ and $\omega=-1.0$, the solution of (\ref{sys.Prabhakar}) in the scalar case, for three distinct values of the coefficient $A\in\mathbb{C}$. 

The three values of the coefficient $A$ are selected, respectively, just inside, on the border and just outside the stability region $S_{\alpha,\beta,\omega}^\gamma$ determined by Theorem \ref{thm.stab.region} and  depicted in the left plot of Figure \ref{Fig:Fig_StabilityRegion_01}. Since the three values of $A$ are almost indistinguishable in the small box in the first quadrant,  an enlarged view of this box is provided in the right plot of the same Figure \ref{Fig:Fig_StabilityRegion_01}.

\begin{figure}[htbp]
    \centering
    \includegraphics[width=0.95\linewidth]{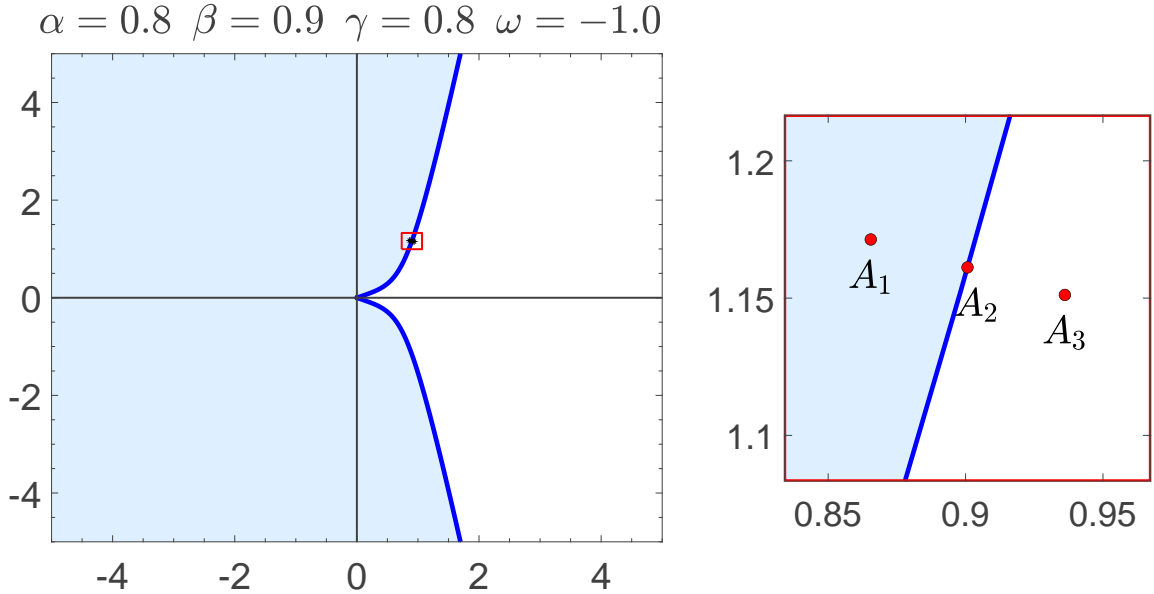}
    \caption{Stability region $S_{\alpha,\beta,\omega}^\gamma$ for $\alpha=0.8$, $\beta=0.9$, $\gamma=0.8$ and $\omega=-1$ (left plot) and zoom of the box with the three values $A_1$, $A_2$, $A_3$ near the border of the stability region (right plot).}
    \label{Fig:Fig_StabilityRegion_01}
\end{figure}

To observe the asymptotic behavior of the solution of  (\ref{sys.Prabhakar}) we have considered both the asymptotic expansion (for large arguments) and the numerical method devised in the previous section. For large $t$ the two approaches provide overlapping results, thus showing their reliability. We therefore  report, in the following plots, only the outcomes from the numerical method which hold for small and large $t$.

The first plot illustrates the solution of (\ref{sys.Prabhakar}) with the coefficient  $A_1 = 0.866  + 1.171\iu$ inside the stability region. As expected from the theory, the solution illustrated in Figure \ref{Fig:Fig_StabilityRegion_A1} shows a stable  behaviour decaying to zero. 

\begin{figure}[htbp]
    \centering
    \includegraphics[width=0.7\linewidth]{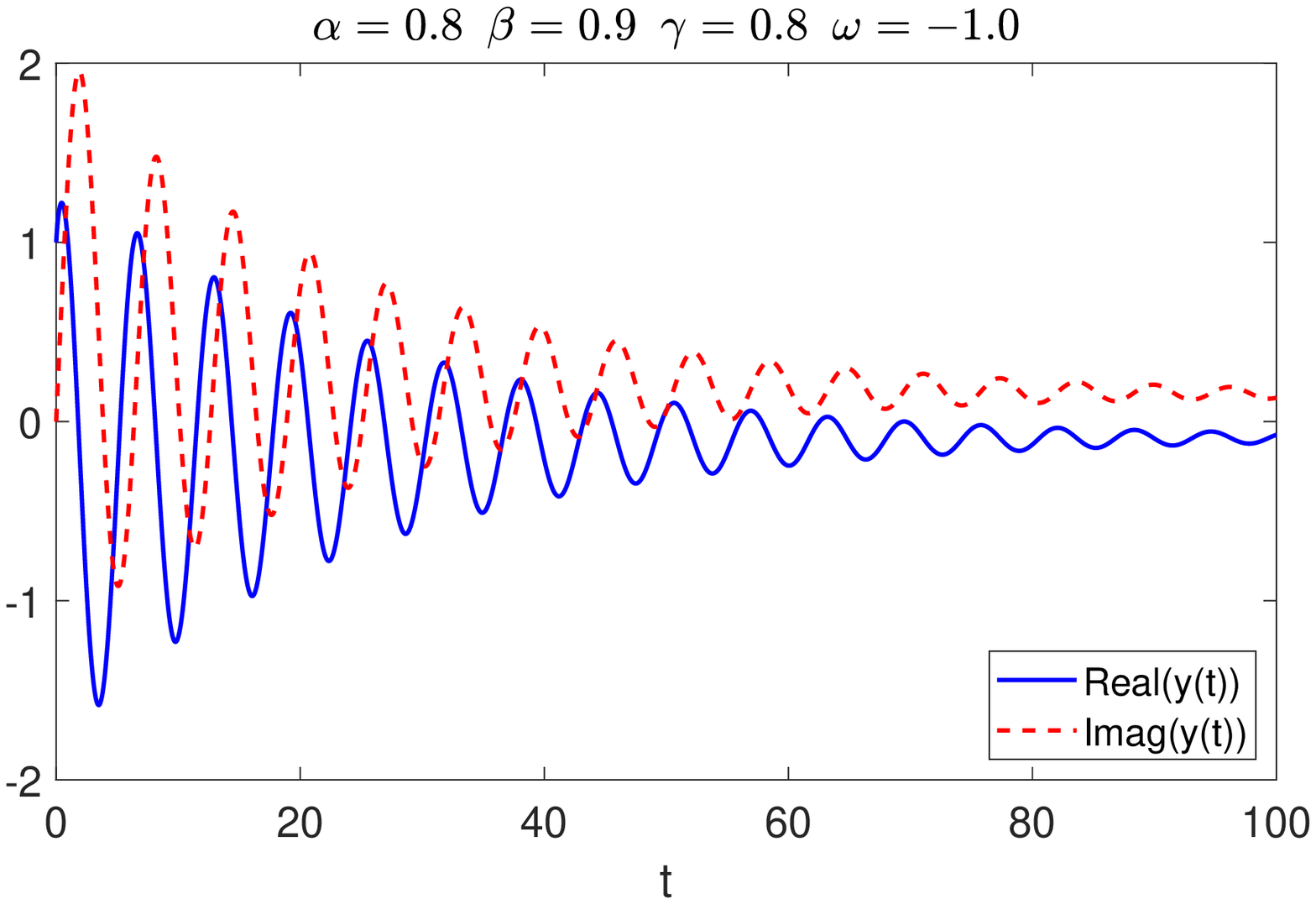}
    \caption{Solution of the linear scalar equation (\ref{sys.Prabhakar}) with $A_1=0.866  + 1.171\iu$ inside the stability region.}
    \label{Fig:Fig_StabilityRegion_A1}
\end{figure}

The second plot shows the solution of the same problem when the coefficient $A_2 = 0.901  + 1.161\iu$ is instead used. Since $A_2$ is on the border of the stability region we expect that, after a transient phase, the solution presents sustained oscillations which neither decay nor amplify. This behaviour is indeed confirmed by the numerical experiment reported in Figure \ref{Fig:Fig_StabilityRegion_A2}.

\begin{figure}[htbp]
    \centering
    \includegraphics[width=0.7\linewidth]{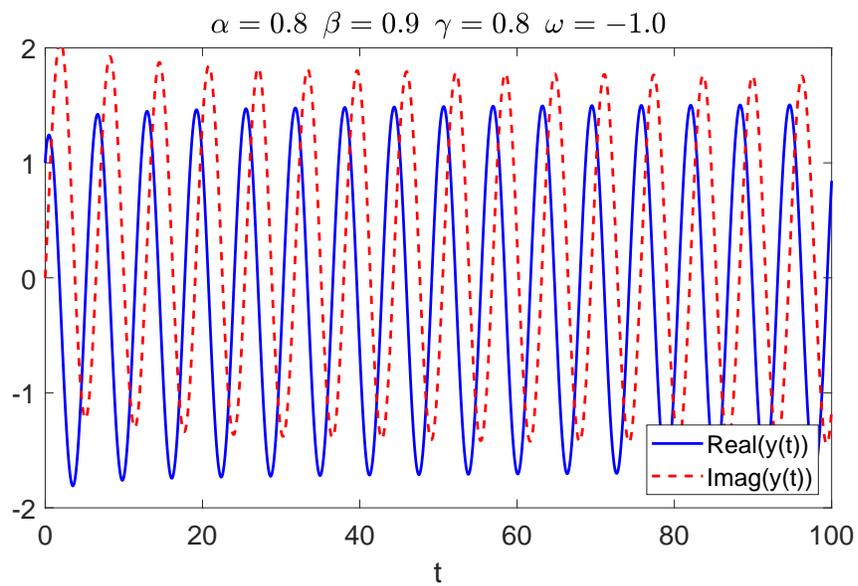}
    \caption{Solution of the linear scalar equation (\ref{sys.Prabhakar}) with $A_2=0.901  + 1.161\iu$ on the border of the stability region.}
    \label{Fig:Fig_StabilityRegion_A2}
\end{figure}

Finally the third experiment concerns the coefficient $A_3 = 0.936  + 1.151\iu$ outside the stability region. In accordance with theoretical expectations, the plot in Figure \ref{Fig:Fig_StabilityRegion_A3} shows an unstable solution with oscillations of growing amplitude as $t$ increases.

\begin{figure}[htbp]
    \centering
    \includegraphics[width=0.7\linewidth]{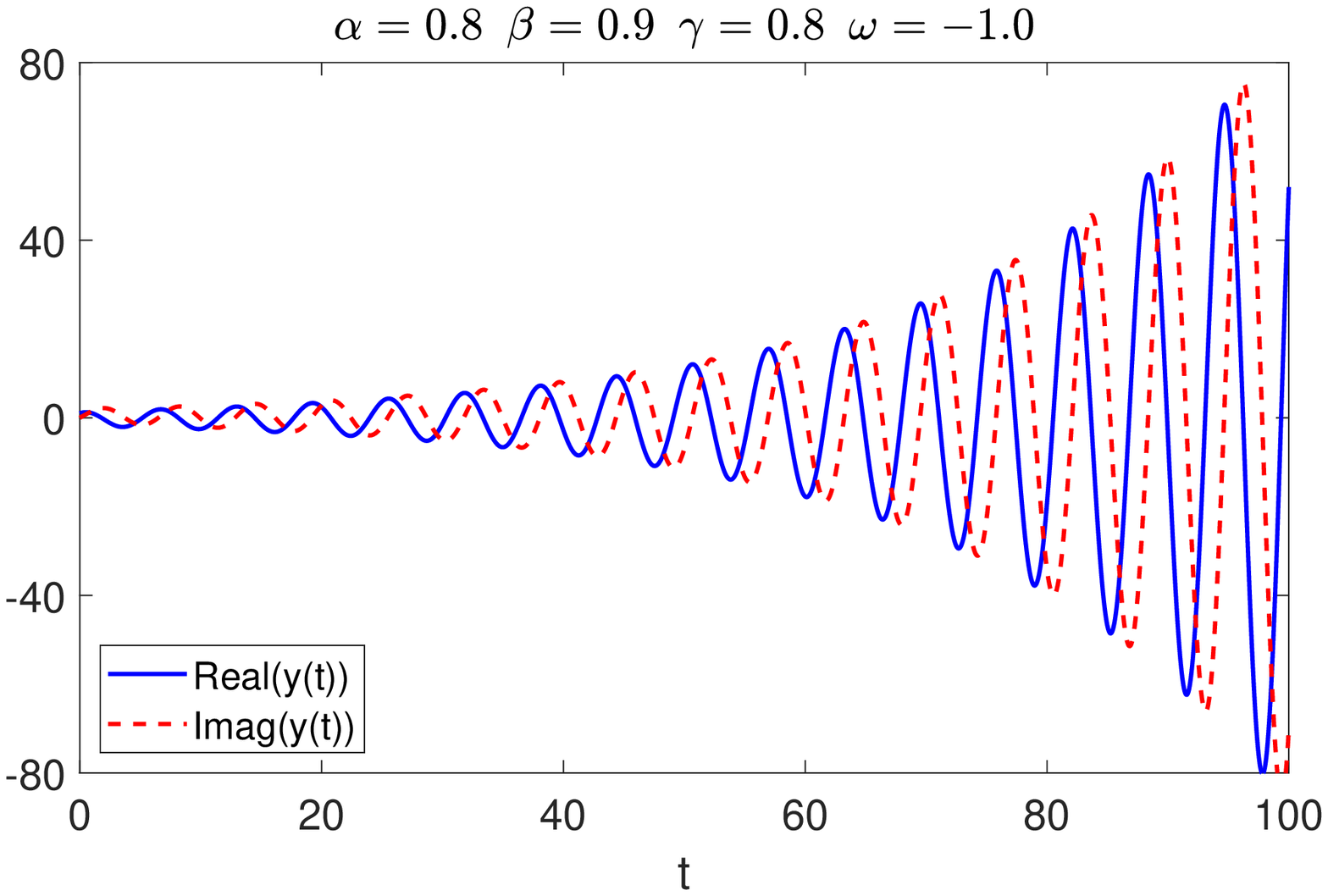}
    \caption{Solution of the linear scalar equation (\ref{sys.Prabhakar}) with $A_3 = 0.936  + 1.151\iu$ outside the stability region.}
    \label{Fig:Fig_StabilityRegion_A3}
\end{figure}

\subsection{A nonlinear example}

It is of interest to provide an example of an application to nonlinear systems of the theoretical results on the asymptotic stability of linear systems of differential equations with the Prabhakar derivative. 

To this purpose we consider a dynamical system of Brusselator type 
\begin{equation}\label{eq:BrusselatorPrabhakar}
\begin{cases}    
    \DP^\gamma_{\alpha,\beta,\omega} x(t)=1 - (b+1)x(t)  + a x(t)^2 y(t) \\
    \DP^\gamma_{\alpha,\beta,\omega}y(t)=bx(t)-a x^2(t)y(t)\
 \end{cases}   ,
\end{equation}
describing an autocatalytic and oscillating chemical reaction, with the integer-order derivative replaced by the fractional Prabhakar derivative.  

It is easy to compute that $(1,b/a)$ is the equilibrium of the system and the eigenvalues of the Jacobian evaluated at the equilibrium point are
\[
	\lambda_{1,2} = \frac{b-a-1 \pm \sqrt{(b-a-1)^2-4a}}{2} .
\]

Let us take into consideration the values of the coefficients $a=10$ and $b=14$ for which the corresponding eigenvalues are $\lambda_{1,2} = 1.5 \pm 2.7839\iu$. 

Depending on the choice of parameters $\alpha$, $\beta$, $\gamma$ and $\omega$ of the Prabhakar derivative, the eigenvalues $\lambda_1$ and $\lambda_2$  can lay inside or outside the corresponding stability region $S_{\alpha,\beta,\omega}^\gamma$. In Figure \ref{Fig:Fig_PrabStab_Bruss_Reg} we show the location of $\lambda_1$ and $\lambda_2$ with respect to the stability region of the Prabhakar derivative when $\alpha=0.9$, $\beta=0.95$ and $\gamma=0.8$ and  $\omega=-4.0$ (left plot) or  $\omega=-0.5$ (right plot).

\begin{figure}[htbp]
    \centering
		\begin{tabular}{cc}
    \includegraphics[width=0.49\linewidth]{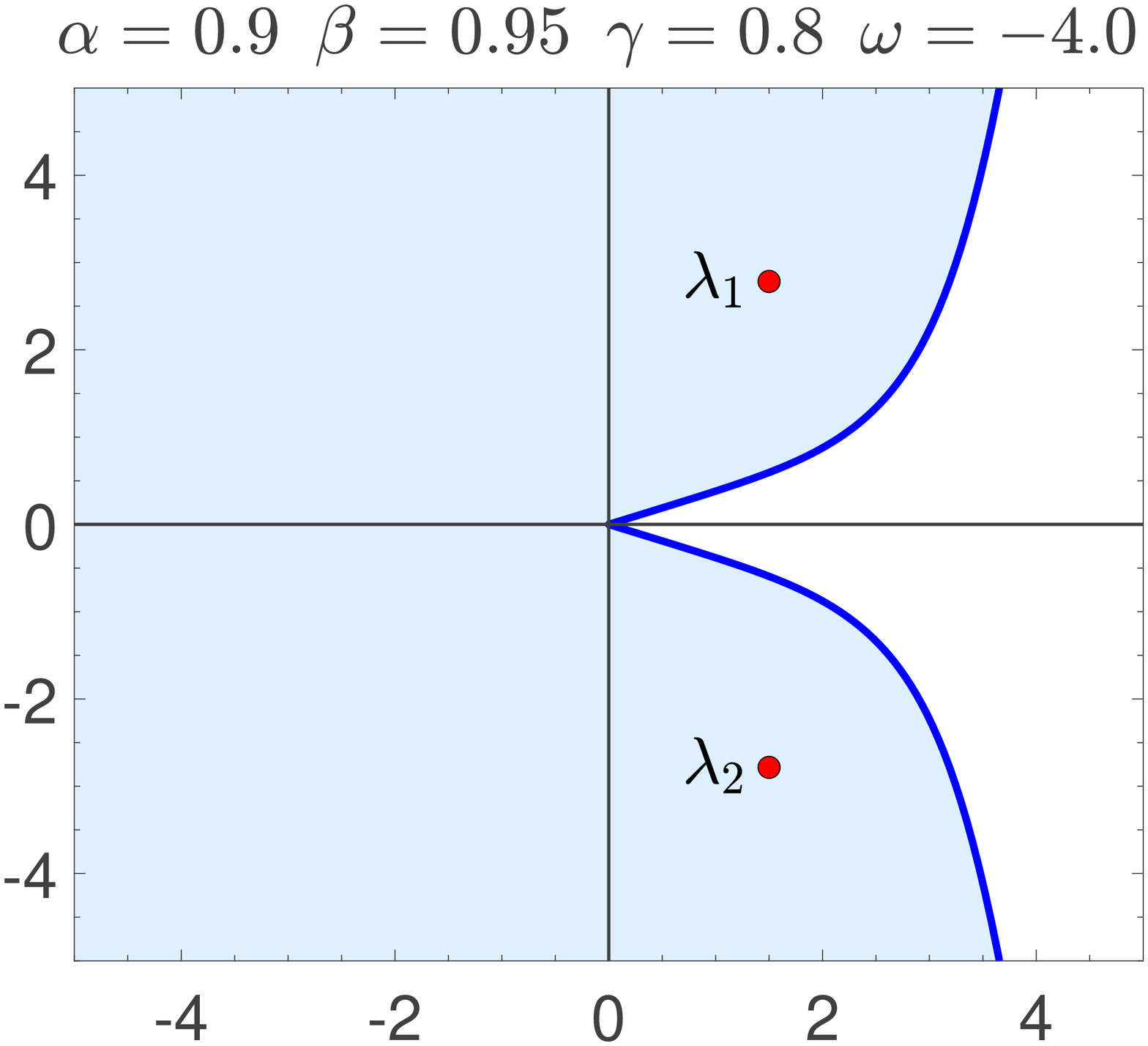} &
    \includegraphics[width=0.49\linewidth]{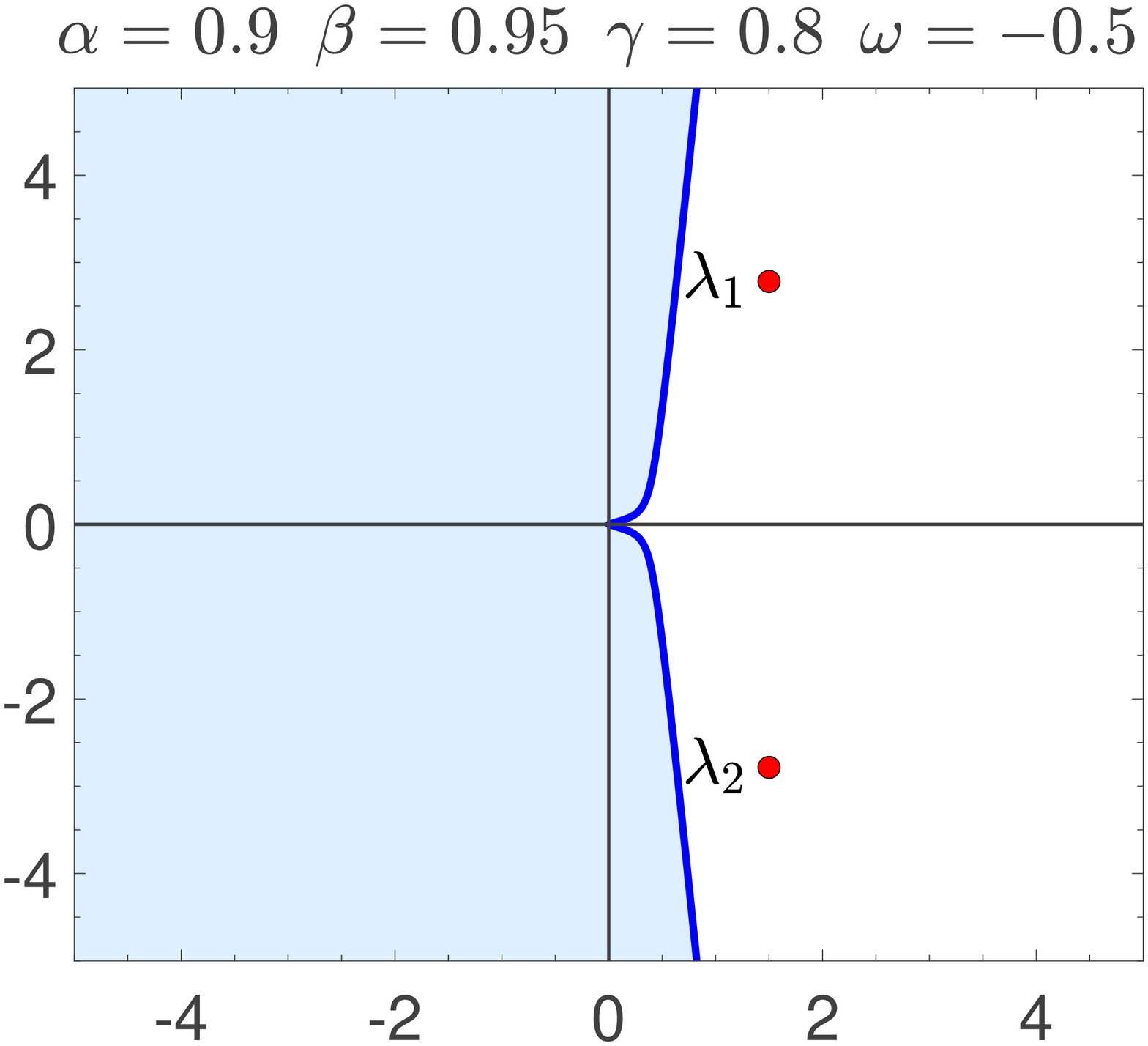} \		
		\end{tabular}
    \caption{Stability regions for $\alpha=0.9$, $\beta=0.95$, $\gamma=0.8$ and $\omega=-4.0$ (left plot) or $\omega=-0.5$ (right plot) and location of the eigenvalues $\lambda_{1,2}$ of the linearized Brusselator system.}
    \label{Fig:Fig_PrabStab_Bruss_Reg}
\end{figure}

The solution of the system (\ref{eq:BrusselatorPrabhakar}) when $\omega=-4.0$, namely when $\lambda_{1,2}$ lie inside the stability region, is shown in Figure \ref{Fig:Fig_PrabStab_Burss_Sol_01}. The two components $x(t)$ and $y(t)$ of the solution approach the equilibrium state (the dotted line), although in a quite slow way, in accordance with the behaviour expected from the asymptotic stability theory.

\begin{figure}[htbp]
    \centering
    \includegraphics[width=0.7\linewidth]{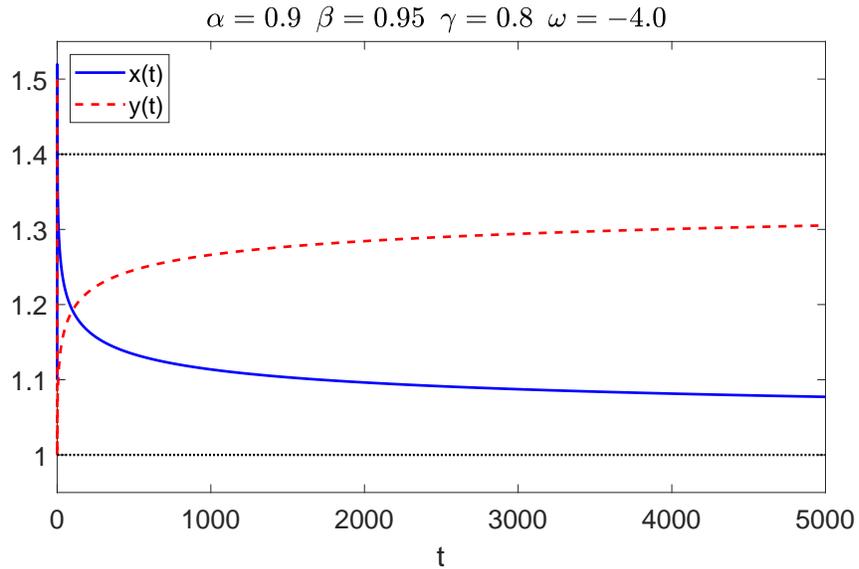}
    \caption{Solution of the Brusselator system (\ref{eq:BrusselatorPrabhakar}) for $\lambda_1$ and $\lambda_2$ in the stability region.}
    \label{Fig:Fig_PrabStab_Burss_Sol_01}
\end{figure}

When $\omega=-0.5$, and hence $\lambda_{1,2}$ are outside the stability region, the equilibrium point is instead unstable and, indeed, the solution of  (\ref{eq:BrusselatorPrabhakar}) oscillates  without ever reaching the equilibrium point as shown in Figure \ref{Fig:Fig_PrabStab_Burss_Sol_02}.

\begin{figure}[htbp]
    \centering
    \includegraphics[width=0.7\linewidth]{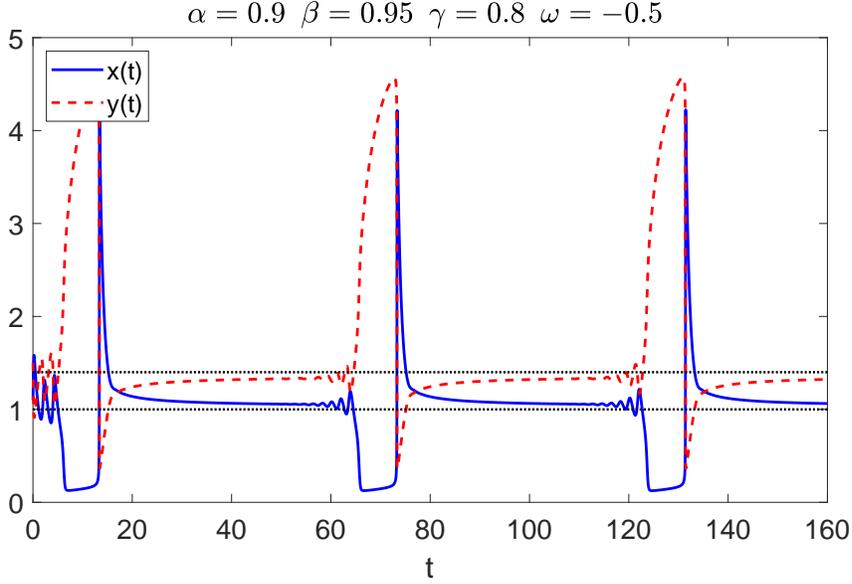}
    \caption{Solution of the Brusselator system (\ref{eq:BrusselatorPrabhakar}) for $\lambda_1$ and $\lambda_2$ outside the stability region.}
    \label{Fig:Fig_PrabStab_Burss_Sol_02}
\end{figure}

In fact, fixing the parameters $\alpha=0.9$, $\beta=0.95$ and $\gamma=0.8$ and numerically solving equation $\Lambda(\theta)=\lambda_1$,
where $\Lambda(\theta)$ defines the parametric equation of the curve $\Psi_{\alpha,\beta,\omega}^\gamma$ as given in Lemma \ref{lem.Lambda}, we obtain the critical value $\omega^\star=-1.58444$, in correspondence of which the eigenvalues $\lambda_{1,2}$ belong to the boundary of the stability region $S_{\alpha,\beta,\omega^\star}^\gamma$. We may consider that in this case, the critical value $\omega^\star=-1.58444$ of the parameter $\omega$ corresponds to a Hopf bifurcation in the Brusselator-type system \eqref{eq:BrusselatorPrabhakar}, resulting in the loss of asymptotic stability of the equilibrium for $\omega>\omega^\star$ and the appearance of an attracting quasi-periodic orbit, as shown in Figure \ref{Fig:Fig_PrabStab_Burss_Sol_02}. However, we emphasize that to the best of our knowledge, at present, the bifurcation theory of fractional-order differential equations with Prabhakar derivatives has not been investigated.

\section{Concluding remarks}

In this paper we have studied asymptotic stability properties of systems of fractional differential equations with the Prabhakar derivative. A complete characterisation of the stability region was obtained, in terms of the eigenvalues of the system's matrix, thus generalizing classical results for the stability of fractional-order systems.

We have also obtained the asymptotic representations of the solution of linear fractional Prabhakar differential equations for small and large arguments and we have presented a numerical method for (linear and nonlinear) differential equations of fractional Prabhakar type.

Some numerical experiments using the asymptotic expansion and the numerical method have been presented in order to validate the theoretical findings. An application to the study of a nonlinear system has also been discussed.



\bibliographystyle{spmpsci}
\bibliography{Prabhakar_biblio}

\end{document}